\documentclass[11pt,letterpaper,reqno]{amsart}

\usepackage[T1]{fontenc}
\usepackage{amsmath,amssymb,amsthm,amsfonts,mathtools}
\usepackage{enumitem}
\usepackage{hyperref}
\hypersetup{colorlinks=true,linkcolor=blue,citecolor=blue,urlcolor=blue}

\newtheorem{theorem}{Theorem}[section]
\newtheorem{lemma}[theorem]{Lemma}

\newtheorem{corollary}[theorem]{Corollary}
\newtheorem{problem}[theorem]{Problem}
\newtheorem{conjecture}[theorem]{Conjecture}
\theoremstyle{definition}

\newtheorem{remark}[theorem]{Remark}
\numberwithin{equation}{section}

\newcommand{\tr}{\operatorname{tr}}
\newcommand{\rank}{\operatorname{rank}}
\newcommand{\1}{\mathbf{1}}
\begin{document}

\title[The third adjacency eigenvalue]{A sharp upper bound on the third adjacency eigenvalue of a graph}

\author[Q.~Tang]{Quanyu Tang}
\address{School of Mathematics and Statistics, Xi'an Jiaotong University, Xi'an 710049, P.~R.~China}
\email{tang\_quanyu@163.com}



\begin{abstract}
For a graph $G$ of order $n$, let
$$
\lambda_1(G)\ge \cdots \ge \lambda_n(G)
$$
be the eigenvalues of its adjacency matrix. We prove that every graph $G$ on $n\ge 3$ vertices satisfies
$$
\lambda_3(G)\le \frac{n}{3}-1,
$$
thereby solving a problem of Nikiforov. The bound is best possible whenever $3\mid n$. Our proof is derived from a more general matrix result: if $A=(a_{ij})$ is a real symmetric matrix of order $n$ with $0\le a_{ij}\le 1$ for all off-diagonal entries and $a_{ii}\ge 0$ for all $i$, then
$$
\lambda_{n-1}(A)+\lambda_n(A)\ge -\frac{2n}{3}.
$$
This in particular confirms a conjecture of Leonida and Li.
\end{abstract}

\maketitle

\section{Introduction}

Throughout, all graphs are finite, simple, and undirected. For a graph $G$ of order $n$, let $A(G)$ denote its adjacency matrix and let
\[
\lambda_1(G)\ge \cdots \ge \lambda_n(G)
\]
be the eigenvalues of $A(G)$. Order-only bounds for adjacency eigenvalues have a long history; see, for example,
Hong~\cite{Hong1988,Hong1993}, Powers~\cite{Powers1989}, and Nikiforov~\cite{Nikiforov2015}.

In~\cite{Powers1989}, Powers proposed, for connected graphs, the order-only bound
\[
\lambda_i(G)\le \Bigl\lfloor \frac{n}{i}\Bigr\rfloor
\qquad (1\le i\le n/2).
\]
As pointed out later by Nikiforov~\cite{Nikiforov2015}, the original proof is flawed. Nevertheless, the bound is valid for $i\le 2$: the case $i=1$ is trivial, and Hong proved the stronger estimate \(\lambda_2(G)\le (n-2)/2\) for every graph of order $n$~\cite{Hong1988}. By contrast, Nikiforov's constructions show that Powers's bound fails for every $i\ge 5$~\cite{Nikiforov2015}, while Linz~\cite{Linz2023} ruled out the case $i=4$ by constructing graphs with $\lambda_4(G)>n/4$. Since connectivity plays no role for $i=3$, the only remaining positive case may be stated as follows; see also \cite[Problem~10]{LiuNing2023} and \cite[Question~2.11]{Nikiforov2015}.

\begin{problem}[Nikiforov~\cite{Nikiforov2015}]\label{prob:powers}
Let $G$ be a graph of order $n\geq 3$. Is it true that
\[
\lambda_3(G)\le \Bigl\lfloor \frac{n}{3}\Bigr\rfloor ?
\]
\end{problem}

Recently, Leonida and Li~\cite[Theorem~1.6]{LeonidaLi2026} verified Problem~\ref{prob:powers} for several important classes of graphs, including strongly regular graphs, regular line graphs, and Cayley graphs on abelian groups. See also Li~\cite{Li2025} for a related asymptotic improvement on Nikiforov's bound.

In order to attack Problem~\ref{prob:powers}, Leonida and Li~\cite[Conjecture~4.2]{LeonidaLi2026} formulated a weighted analogue. Let $\mathcal{S}_n$ denote the set of real symmetric $n\times n$ matrices with entries in $[0,1]$.

\begin{conjecture}[Leonida--Li~\cite{LeonidaLi2026}]\label{conj:LL}
One has
\[
\inf_{M\in \mathcal{S}_n}\lambda_{n-1}(M)\ge -\frac{n}{3}.
\]
\end{conjecture}

The main result of this paper is a complete solution to Conjecture~\ref{conj:LL}. In fact, we prove a slightly stronger statement, allowing arbitrary nonnegative diagonal entries.

\begin{theorem}\label{thm:weighted}
Let $n\ge 2$, and let $A=(a_{ij})$ be a real symmetric matrix of order $n$ with eigenvalues
\[
\mu_1\ge \mu_2\ge \cdots \ge \mu_n.
\]
Assume that
\[
0\le a_{ij}\le 1\quad (i\ne j),
\qquad
a_{ii}\ge 0\quad (1\le i\le n).
\]
Then
\[
\mu_{n-1}+\mu_n\ge -\frac{2n}{3}.
\]
In particular, \(\mu_{n-1}\ge -\frac{n}{3}\).
\end{theorem}

Combining Theorem~\ref{thm:weighted} with the identity \(A(G)+A(\overline G)=J-I\) and Weyl's inequality, we obtain the following graph-theoretic consequence. In particular, it yields a stronger conclusion than Problem~\ref{prob:powers}.

\begin{theorem}\label{thm:main}
Let $G$ be a graph on $n\ge 3$ vertices. Then
\[
\lambda_3(G)\le \frac{n}{3}-1.
\]
\end{theorem}

\begin{remark}
The bound in Theorem~\ref{thm:main} is sharp whenever $3\mid n$. Indeed,
\[
\lambda_3\left(K_{n/3}\cup K_{n/3}\cup K_{n/3}\right)=\frac{n}{3}-1.
\]
More generally, Leonida and Li~\cite[Theorem~2.2]{LeonidaLi2026} constructed a family of graphs $H_{a,b}$ of order $3(a+b)$ satisfying
\[
\lambda_3(H_{a,b})=a+b-1=\frac{|V(H_{a,b})|}{3}-1.
\]
These provide nonregular equality examples, and the closed blow-ups of this family still have third-eigenvalue proportion tending to $1/3$.
\end{remark}

The paper is organised as follows. In Section~\ref{sec:weighted} we prove Theorem~\ref{thm:weighted}. In Section~\ref{sec:graph} we deduce Theorem~\ref{thm:main}.

\section{A rank-two projection bound}\label{sec:weighted}

The proof begins with a simple trigonometric majorant.

\begin{lemma}\label{lem:trig}
For every real number $x$,
\[
|\cos x|\le \frac{2}{3}+\frac{7}{18}\cos 2x-\frac{1}{18}\cos 4x.
\]
\end{lemma}

\begin{proof}
Set $t=|\cos x|\in[0,1]$. Then
\[
\cos 2x=2t^2-1,
\qquad
\cos 4x=8t^4-8t^2+1.
\]
Hence
\begin{align*}
\frac{2}{3}+\frac{7}{18}\cos 2x-\frac{1}{18}\cos 4x-|\cos x|
&=\frac{2}{3}+\frac{7}{18}(2t^2-1)-\frac{1}{18}(8t^4-8t^2+1)-t\\
&=\frac{(1-t)(t+2)(2t-1)^2}{9}\ge 0.
\end{align*}
This proves the claim.
\end{proof}

\begin{proof}[Proof of Theorem~\ref{thm:weighted}]
Let
\[
\mathcal P_2:=\{Q\in\mathbb R^{n\times n}: Q^2=Q,\ Q^T=Q,\ \rank Q=2\}
\]
be the set of orthogonal projections of rank $2$. By Ky Fan's minimum principle (see, e.g.,~\cite[Corollary~4.3.39]{HornJohnson2012}),
\begin{equation}\label{eq:kyfan}
\mu_{n-1}+\mu_n=\min_{Q\in\mathcal P_2}\tr(AQ).
\end{equation}
Thus it suffices to prove that
\[
\tr(AQ)\ge -\frac{2n}{3}
\]
for every $Q\in\mathcal P_2$.

Fix $Q\in\mathcal P_2$. Write
\[
Q=RR^T,
\qquad
R\in\mathbb R^{n\times 2},
\qquad
R^TR=I_2.
\]
Let the $i$th row of $R$ be $r_i\in\mathbb R^2$, and write \(q_{ij}=r_i\cdot r_j\). Set $c_i=\sqrt{r_{i1}^2+r_{i2}^2}\ge 0$. If $c_i>0$, choose $\theta_i\in\mathbb R$ such that \(r_i=c_i(\cos\theta_i,\sin\theta_i)\), and if $c_i=0$, choose $\theta_i$ arbitrarily. Then
\begin{equation}\label{eq:qij}
q_{ij}=c_ic_j\cos(\theta_i-\theta_j).
\end{equation}

Define complex numbers
\[
z_j:=r_{j1}+ir_{j2}=c_j e^{i\theta_j}.
\]
Since $R^TR=I_2$, we have
\[
\sum_{i=1}^n r_{i1}^2=1,
\qquad
\sum_{i=1}^n r_{i2}^2=1,
\qquad
\sum_{i=1}^n r_{i1}r_{i2}=0.
\]
Therefore,
\begin{equation}\label{eq:frameid}
\sum_{i=1}^n c_i^2=\sum_{i=1}^n |z_i|^2=2,
\qquad
\sum_{i=1}^n c_i^2e^{2i\theta_i}=\sum_{i=1}^n z_i^2=0.
\end{equation}

We now bound the entrywise $\ell_1$-norm of $Q=(q_{ij})$. Set
\[
C:=\sum_{i=1}^n c_i,
\qquad
S:=\sum_{i=1}^n c_ie^{2i\theta_i},
\qquad
T:=\sum_{i=1}^n c_ie^{4i\theta_i}.
\]
By \eqref{eq:qij} and Lemma~\ref{lem:trig},
\begin{align}
\sum_{i,j=1}^n |q_{ij}|
&=\sum_{i,j=1}^n c_ic_j|\cos(\theta_i-\theta_j)|\notag\\
&\le \frac{2}{3}C^2
+\frac{7}{18}\sum_{i,j=1}^n c_ic_j\cos 2(\theta_i-\theta_j)
-\frac{1}{18}\sum_{i,j=1}^n c_ic_j\cos 4(\theta_i-\theta_j).
\label{eq:l1start}
\end{align}
But
\[
\sum_{i,j=1}^n c_ic_je^{2i(\theta_i-\theta_j)}
=\left(\sum_{i=1}^n c_ie^{2i\theta_i}\right)
\overline{\left(\sum_{j=1}^n c_je^{2i\theta_j}\right)}
=|S|^2,
\]
which is real, so
\[
\sum_{i,j=1}^n c_ic_j\cos 2(\theta_i-\theta_j)=|S|^2.
\]
Similarly,
\[
\sum_{i,j=1}^n c_ic_j\cos 4(\theta_i-\theta_j)=|T|^2.
\]
Hence \eqref{eq:l1start} becomes
\begin{equation}\label{eq:l1reduced}
\sum_{i,j=1}^n |q_{ij}|
\le \frac{2}{3}C^2+\frac{7}{18}|S|^2-\frac{1}{18}|T|^2
\le \frac{2}{3}C^2+\frac{7}{18}|S|^2.
\end{equation}

To estimate $C^2+|S|^2$, define $M\in\mathbb C^{n\times 2}$ by declaring its $j$th row to be
\[
m_j:=\frac{c_j}{\sqrt 2}(1,e^{2i\theta_j}).
\]
Let \(M^*\) denote the conjugate transpose of \(M\). Using \eqref{eq:frameid}, we obtain
\[
M^*M
=\frac12
\begin{pmatrix}
\sum_i c_i^2 & \sum_i c_i^2e^{2i\theta_i}\\
\sum_i c_i^2e^{-2i\theta_i} & \sum_i c_i^2
\end{pmatrix}
=I_2.
\]
Hence, for every \(x\in\mathbb C^n\), where \(\|\cdot\|_2\) denotes the
Euclidean norm on complex vectors, we have
\[
x^*MM^*x=(M^*x)^*(M^*x)=\|M^*x\|_2^2\le \|x\|_2^2=x^*I_nx.
\]
Therefore \(MM^*\preceq I_n\), where \(\preceq\) denotes the usual Loewner
order on Hermitian matrices. Let \(\1\in\mathbb R^n\subset \mathbb C^n\) be
the all-ones vector. Then
\[
\|M^*\1\|_2^2=\1^*MM^*\1\le \|\1\|_2^2=n.
\]
On the other hand,
\[
M^*\1=\frac{1}{\sqrt 2}
\binom{\sum_i c_i}{\sum_i c_ie^{-2i\theta_i}},
\]
so
\begin{equation}\label{eq:CSbound}
\frac12(C^2+|S|^2)=\|M^*\1\|_2^2\le n.
\end{equation}
Combining \eqref{eq:l1reduced} with \eqref{eq:CSbound}, we get
\begin{align}
\sum_{i,j=1}^n |q_{ij}|
&\le \frac{2}{3}C^2+\frac{7}{18}|S|^2\notag\\
&\le \frac{2}{3}(C^2+|S|^2)\notag\\
&\le \frac{2}{3}\cdot 2n
=\frac{4n}{3}.
\label{eq:l1final}
\end{align}

Since $Q$ is a projection of rank $2$, we have $\tr Q=2$ and $q_{ii} = \|r_i\|_2^2 \ge 0$ for all $i$. Thus
\begin{equation}\label{eq:offdiagabs}
\sum_{1\le i<j\le n}|q_{ij}|
=\frac12\left(\sum_{i,j=1}^n |q_{ij}|-\sum_{i=1}^n q_{ii}\right)
\le \frac12\left(\frac{4n}{3}-2\right)=\frac{2n}{3}-1.
\end{equation}
Also,
\[
\1^TQ\1=\1^TRR^T\1=\|R^T\1\|_2^2\ge 0.
\]
Thus
\begin{equation}\label{eq:offdiagsum}
\sum_{1\le i<j\le n}q_{ij}
=\frac12(\1^TQ\1-\tr Q)
\ge \frac12(0-2)=-1.
\end{equation}
Hence
\begin{align}
\sum_{1\le i<j\le n}\min(q_{ij},0)
&=\frac12\left(\sum_{i<j}q_{ij}-\sum_{i<j}|q_{ij}|\right)\notag\\
&\ge \frac12\left(-1-\left(\frac{2n}{3}-1\right)\right)
=-\frac{n}{3}.
\label{eq:minneg}
\end{align}

Finally, since $0\le a_{ij}\le 1$ for $i\ne j$, we have
\[
a_{ij}q_{ij}\ge \min(q_{ij},0)
\qquad (i<j),
\]
and because $a_{ii}\ge 0$ and $q_{ii}\ge 0$, we also have $a_{ii}q_{ii}\ge 0$. Therefore,
\begin{align*}
\tr(AQ)
&=\sum_{i=1}^n a_{ii}q_{ii}+2\sum_{1\le i<j\le n}a_{ij}q_{ij}\\
&\ge 2\sum_{1\le i<j\le n}\min(q_{ij},0)\\
&\ge -\frac{2n}{3}
\end{align*}
by \eqref{eq:minneg}. Since $Q\in\mathcal P_2$ was arbitrary, \eqref{eq:kyfan} yields
\[
\mu_{n-1}+\mu_n\ge -\frac{2n}{3}.
\]
Because $\mu_n\le \mu_{n-1}$, we conclude that
\[
\mu_{n-1}\ge \frac{\mu_{n-1}+\mu_n}{2}\ge -\frac{n}{3}.
\]
This completes the proof.
\end{proof}

\begin{corollary}\label{cor:graphlower}
Let $H$ be a graph on $n\geq 2$ vertices. Then
\[
\lambda_{n-1}(H)\ge -\frac{n}{3}.
\]
\end{corollary}

\begin{proof}
Apply Theorem~\ref{thm:weighted} to the adjacency matrix $A(H)$.
\end{proof}

\section{Deduction of the graph-theoretic theorem}\label{sec:graph}

\begin{proof}[Proof of Theorem~\ref{thm:main}]
Let $\overline G$ be the complement of $G$. Then
\[
A(G)+A(\overline G)=J-I,
\]
where $J$ is the all-ones matrix. The eigenvalues of $J-I$ are $n-1$ and $-1$ with multiplicity $n-1$. In particular,
\[
\lambda_2(J-I)=-1.
\]

We apply Weyl's inequality for Hermitian matrices in the form
\[
\lambda_i(X)+\lambda_j(Y)\le \lambda_{i+j-n}(X+Y)
\qquad (i+j\ge n+1),
\]
see, for example,~\cite[\S 4.3]{HornJohnson2012}. Taking
\[
X=A(G),\qquad Y=A(\overline G),\qquad i=3,\qquad j=n-1,
\]
we obtain
\[
\lambda_3(G)+\lambda_{n-1}(\overline G)\le \lambda_2(J-I)=-1.
\]
By Corollary~\ref{cor:graphlower},
\[
\lambda_{n-1}(\overline G)\ge -\frac{n}{3}.
\]
Hence
\[
\lambda_3(G)\le -1-\lambda_{n-1}(\overline G)\le -1+\frac{n}{3}=\frac{n}{3}-1.
\]
This proves the claimed bound.
\end{proof}

\begin{remark}
Theorem~\ref{thm:main} is sharp for all $n$ divisible by $3$.
Determining the exact maximum of $\lambda_3(G)$ for $3\nmid n$ remains an interesting problem.
\end{remark}

\end{document}